\documentclass[11pt]{amsart}
\usepackage{amsmath,graphics}
\usepackage{amsfonts,amssymb, enumerate, bbold, color}
\usepackage{ulem}
\usepackage{verbatim}

\usepackage{fancyhdr}								   					 
\usepackage{lastpage}
\usepackage{mathrsfs}

\usepackage{geometry}
\usepackage{enumerate}								     					
\usepackage{thmtools}
\usepackage{thm-restate}
\usepackage{hyperref}
\usepackage[capitalise]{cleveref}
\usepackage{mathtools}
\usepackage{bbm}									     					
\usepackage{framed}	
\usepackage{eufrak}		

\usepackage[latin1]{inputenc}

\usepackage[T1]{fontenc}
\usepackage[all,cmtip]{xy}

\usepackage{lipsum}
\newcounter{quotecount}

\usepackage{eucal} 
\usepackage{cmmib57} 

\usepackage{graphicx}								  					 
\usepackage{float}									    					 
\usepackage{epstopdf}
\usepackage{pgfplots}
\pgfplotsset{compat=1.12}	
\usepackage{subcaption}	
\usetikzlibrary{snakes}	
\usepackage{tikz-cd}			    
\tikzset{snake arrow/.style=
{-,
decorate,
decoration={snake,amplitude=.2mm,segment length=1mm,post length=0mm}}
}

\theoremstyle{plain}
\newtheorem{theorem}{Theorem}[section]
\newtheorem{lemma}[theorem]{Lemma}
\newtheorem{corollary}[theorem]{Corollary}
\newtheorem{proposition}[theorem]{Proposition}

\newtheorem{remark}[theorem]{Remark}
\newtheorem{definition}[theorem]{Definition}
\newtheorem{notation}[theorem]{Notation}
\newtheorem{example}[theorem]{Example}

\newtheorem*{ack}{Acknowledgement}
\theoremstyle{definition}
\newtheorem{defn}[theorem]{Definition}

\newcommand{\PC}{{\mathcal P}}

\renewcommand{\Re}{\operatorname{Re}}
\renewcommand{\Im}{\operatorname{Im}}

\renewcommand{\Im}{\text{Im}}

\newcommand{\rank}{\operatorname{rank}}					

\newcommand{\cL}{\mathcal{L}}

\newcommand{\C}{\mathbb{C}}

\newcommand{\Q}{\mathbb{Q}}
\newcommand{\R}{\mathbb{R}}

\newcommand{\Z}{\mathbb{Z}}

\newcommand{\h}{\mathfrak{h}}

\newcommand{\s}{\mathfrak{S}}

\newcommand{\re}{\text{Re}}											
\newcommand{\vepsi}{\varepsilon}
\newcommand{\Qt}{\Q[t^{\pm 1}]}				
\newcommand{\f}[5]{
\begin{array}{rcl}
#1:#2 & \longrightarrow & #3 \\
#4 & \longmapsto & #5  \\
\end{array}
}						

\DeclareMathOperator{\Aut}{Aut}										
\DeclareMathOperator{\codim}{codim}		

\def\part{\partial}

\def\ker{\mbox{Ker}}

\def\lra{\longrightarrow}

\def\bd{\begin{definition}}
\def\ed{\end{definition}}
\def\bt{\begin{theorem}}
\def\et{\end{theorem}}
\def\br{\begin{remark}}
\def\er{\end{remark}}
\def\bc{\begin{corollary}}
\def\ec{\end{corollary}}
\def\bp{\begin{proposition}}
\def\ep{\end{proposition}}
\def\be{\begin{equation}}
\def\ee{\end{equation}}
\def\bn{\begin{enumerate}}
\def\en{\end{enumerate}}
\def\ba{\begin{array}}
\def\ea{\end{array}}
\def\bl{\begin{lemma}}
\def\el{\end{lemma}}

\def\bex{\begin{example}}
\def\eex{\end{example}}

\begin{document}

\title[]{On the signed Euler characteristic property for subvarieties of abelian varieties}
\author{Eva Elduque}
\address{Department of Mathematics,
          University of Wisconsin-Madison,
          480 Lincoln Drive, Madison WI 53706-1388, USA.}
\email {evaelduque@math.wisc.edu}
\author{Christian Geske}
\address{Department of Mathematics,
          University of Wisconsin-Madison,
          480 Lincoln Drive, Madison WI 53706-1388, USA.}
\email {cgeske@math.wisc.edu}
\author{Laurentiu Maxim}
\address{Department of Mathematics,
          University of Wisconsin-Madison,
          480 Lincoln Drive, Madison WI 53706-1388, USA.}
\email {maxim@math.wisc.edu}

\thanks{}

\date{\today}



\begin{abstract}
We give an elementary proof of the fact that a pure-dimensional closed subvariety of a complex abelian variety has a signed intersection homology Euler characteristic. We also show that such subvarieties which, moreover, are local complete intersections, have a signed Euler-Poincar\'e characteristic. Our arguments rely on the construction of circle-valued Morse functions on such spaces, and use in an essential way the stratified Morse theory of Goresky-MacPherson. Our approach also applies (with only minor modifications) for proving similar statements in the analytic context, i.e., for subvarieties of compact complex tori. Alternative proofs of our results can be given by using the general theory of perverse sheaves. 
\end{abstract}

\maketitle

\tableofcontents

\section{Introduction}

A compact {complex torus} of dimension $g$ is a torus of real dimension $2g$ that carries the structure of a complex manifold. It can always be obtained as the quotient of a $g$-dimensional complex vector space by a lattice of rank $2g$. 
A complex {\it abelian variety} of dimension $g$ is a compact complex torus of dimension $g$ that is also a complex projective variety.  Since it is a complex torus, an abelian variety carries the structure of a group. 
Abelian varieties are among the most studied objects in algebraic geometry and indispensable tools 
in topology and number theory.

A complex {\it semi-abelian variety} is a complex algebraic group $G$ which is an extension
$$1 \to T \to G \to A \to 1,$$
where $A$ is a complex abelian variety and $T\cong(\C^*)^m$ is a complex affine torus. 
One of the most interesting examples of semi-abelian varieties is the {Albanese variety}  
associated to a smooth complex quasi-projective variety. 
A host of problems in algebraic geometry and topology can be reduced to the study of semi-abelian varieties via the corresponding (generalized) {Albanese map}, e.g., see \cite{Iit}.

The departing point for this note is the following result of Franecki-Kapranov: 
\bt\label{th1} \cite{FK}
If $G$ is a semi-abelian variety and $\PC$ is a perverse sheaf on $G$, the Euler characteristic of $\PC$ is non-negative, i.e., \be\label{nnp} \chi(G,\PC)\geq 0.\ee
\et

The proof of (\ref{nnp}) in \cite{FK} is derived from a Riemann-Roch-type formula for constructible sheaf complexes on a semi-abelian variety. The result also follows from 
the generic vanishing theorems for perverse sheaves \cite{Kra,LMWb}, see also \cite{BSS,Sc15,We} for the case of abelian varieties, while the corresponding statement for the Euler characteristic of perverse sheaves on a complex affine torus was proved in \cite{GL,LMWc}.

As important special cases of Theorem \ref{th1}, we mention the following:\\
\noindent $(a)$ If $X$ is a smooth closed  $n$-dimensional subvariety of a complex semi-abelian variety $G$, the topological Euler characteristic of $X$ is {\it signed}, i.e.,
\be\label{e1} (-1)^n \cdot \chi(X) \geq 0.\ee
\noindent $(b)$ If $X$ is a closed subvariety of pure dimension $n$ of a complex semi-abelian variety $G$, which moreover is a local complete intersection, then\be\label{e2} (-1)^n \cdot \chi(X) \geq 0.\ee
\noindent $({c})$ If $X$ is a closed subvariety of pure dimension $n$ of a complex semi-abelian variety $G$, the  intersection homology Euler characteristic of $X$ is {\it signed}, i.e.,
\be\label{e3} (-1)^n \cdot I\chi(X) \geq 0.\ee
Indeed, in each of the above cases, there exists a perverse sheaf that can be extracted from the underlying geometric situation, namely, the shifted constant sheaf $\Q_X[n]$ in cases $(a)$ and $(b)$, and, respectively, the intersection cohomology complex $IC_X$ in case $({c})$.

\medskip

Since the statements of (\ref{e1})-(\ref{e3}) are purely topological, it is natural to attempt to justify them by more elementary methods. This was achieved in \cite{LMWa} in the case $(a)$, where non-proper Morse theoretic techniques developed by Palais-Smale were used to prove (\ref{e1}). The goal of this paper is to give a topological (Morse theoretic) explanation for (\ref{e2}) and (\ref{e3}). We will use methods similar to those developed in \cite{LMWa}, but placed in the context of the {\it stratified Morse theory} of Goresky-MacPherson \cite{GMP}. For technical reasons (properness of Morse functions), we restrict ourselves to the compact case, i.e., we prove (\ref{e2}) and (\ref{e3}) in the case when $X$  is a pure-dimensional closed subvariety of an abelian variety, see Theorems \ref{thmmain} and \ref{thmlci}. In this case, (\ref{e2}) and (\ref{e3}) are direct consequences of the existence of circle-valued Morse functions on $X$, satisfying certain good properties, which we  prove in Theorem \ref{thmexistence}.
We do, however, expect that our arguments extend to the general semi-abelian case by using non-proper stratified Morse theory. Our approach also applies (with only minor modifications) for proving similar statements in the analytic context, i.e., for subvarieties of compact complex tori (see Remark \ref{ana}). 

\medskip

The paper is organized as follows. In Section \ref{IH}, we recall the definition of intersection homology groups via allowability conditions, and derive some preliminary results on the corresponding intersection homology Euler characteristic. In Section \ref{SM}, we review the basic facts about stratified Morse theory. Section \ref{Ex} is the main technical part of the paper, and it is devoted to the construction of circle-valued Morse functions defined on a closed subvariety of an abelian variety. The proofs of (\ref{e2}) and (\ref{e3}) are given in Section \ref{PR}.

\medskip

\begin{ack} We thank Botong Wang, Mois\'es Herrad\'on Cueto and J\"org Sch\"urmann for valuable discussions.
\end{ack}

\section{Intersection homology}\label{IH}
Classically,  much of the manifold theory, e.g., Morse theory, Lefschetz theorems, Hodge decompositions, and especially Poincar\'e Duality, is recovered in the singular stratified context if, instead of the usual (co)homology, one uses Goresky-MacPherson's intersection homology groups \cite{GM1,GM2}. We recall here the definition of intersection homology of complex analytic (or algebraic) varieties, and discuss some preliminary results concerning the corresponding intersection homology Euler characteristic. For more details on intersection homology, the reader may consult, e.g., \cite{Fr} and the references therein.

Let $X$ be a purely $n$-dimensional complex analytic (or algebraic) variety with a fixed Whitney stratification. All strata of $X$ are of even real (co)dimension.
 By \cite{G}, $X$ admits a triangulation which is compatible with the stratification, so $X$ can also be viewed as a PL stratified pseudomanifold. 
Let $(C_*(X),\partial)$ denote the complex of finite PL chains 
on $X$, with $\Z$-coefficients. 
 The {\it intersection homology groups} of $X$, denoted $IH_i(X)$, are the homology groups of a complex of ``allowable chains'', defined by imposing restrictions on how chains meet the singular strata. Specifically, the chain complex $(IC_*(X),\partial)$ of allowable  finite PL chains is defined as follows: if $\xi \in C_i(X)$ has support $\vert \xi \vert$, 
 then $\xi \in IC_i(X)$ if, and only if, 
$$ \dim ( \vert \xi \vert \cap S ) < i - s $$
and  $$\dim ( \vert \partial\xi \vert \cap S ) < i - s -1,$$ for each stratum $S$ of complex codimension $s>0$. The boundary operator on allowable chains is induced from the usual boundary operator on chains of $X$. (The second condition above ensures that $\partial$ restricts to the complex of allowable chains.)

\br\label{r1}\rm The above definition of intersection homology groups corresponds to the so-called ``middle-perversity'', which will be used implicitly throughout this paper. It should also be mentioned that the intersection homology groups can be realized, as in \cite{Bo,GM2}, by the compactly supported hypercohomology of the (shifted) Deligne intersection sheaf complex $IC_X$. In fact, the intersection chain complex $(IC_*(X),\partial)$ is obtained (up to a shift) as $\Gamma_c(X,IC_X)$, i.e., global sections with compact support of the Deligne complex $IC_X$. Note also that although intersection homology is defined with reference to a fixed stratification, it is in fact independent of the choice of stratification and, moreover, it is a topological invariant (where, for proving this fact, the above-mentioned sheaf-theoretic realization of intersection homology is needed, together with the stratification independence of the Deligne complex, see \cite[Section 4.1]{GM2}).
\er

In what follows, we define a twisted version of the above intersection homology groups, which will be needed later on.  Let $\mathscr{L}$ be a {\it globally defined} rank-one $\Qt$-local system, i.e., a local system on $X$ with stalk $\Qt$, corresponding to a representation
$$
\rho:\pi_1(X)\longrightarrow \Aut_{\Qt}(\Qt)\cong\Qt^*
$$
Let $H=\pi_1(X)/\ker (\rho)$, and consider the covering space $X_H\longrightarrow X$ associated to the projection $\pi_1(X) \to H$. Hence the deck group of $X_H$ is isomorphic to $H$. Note that $X_H$ is a complex analytic space (it is not an algebraic variety in general). 
We can lift the stratification and PL-structure on $X$ to a stratification and PL-structure on $X_H$, so that the deck transformations  of $X_H$ are stratum-preserving PL homeomorphisms. It follows that the complex $IC_*(X_H)$ of allowable finite PL chains in $X_H$  
can be seen as a complex of free left modules over $\Z[H]$, where the action is given by deck transformations. 
Let us now define 
$$
IC_*(X,\mathscr{L}):=\Qt\otimes_{\Z[H]}IC_*(X_H)
$$
where the PID $\Qt$ is seen as a right $\Z[H]$-module, with corresponding action given by
$$
p(t)\cdot \alpha=p(t)\cdot \rho(\alpha)
$$
for every $p(t)\in\Qt$ and every $\alpha\in H$.
We define the twisted intersection homology groups $IH_*(X,\mathscr{L})$ as the homology groups associated to the complex $IC_*(X,\mathscr{L})$.  

\medskip

For the remainder of this section, we assume that $X$ is a complex algebraic variety. 
Then the intersection homology groups $IH_i(X)$ are finitely generated (e.g., see 
\cite{Bo}[V.10.13]), so the corresponding intersection homology Euler characteristic $I\chi(X)$ is well-defined by:
$$
I\chi(X)=\sum_{i=0}^{2n}(-1)^n\rank_\Z IH_i(X).$$
Furthermore, intersection homology groups can be computed simplicially, with respect to a {\it full} triangulation $T$ of $X$, i.e., there are group isomorphisms $IH_i(X) \cong IH_i^T(X)$, with $IH_i^T(X)$ denoting the {\it simplicial} intersection homology groups corresponding to the full triangulation $T$, see \cite[Theorem 3.3.20]{Fr}. Such full triangulations of a PL filtered space always exists, e.g., see \cite[Lemma 3.3.19]{Fr}.
\br \rm
Since we use finite chains to define intersection homology, it can be seen by using Remark \ref{r1} that $$(-1)^n \cdot I\chi(X)=\chi_c(X,IC_X),$$ where $\chi_c(X,IC_X)$ denotes the compactly supported Euler characteristic of the constructible complex $IC_X$. However, since it is known that $$\chi_c(X,IC_X)=\chi(X,IC_X)$$ (e.g., see \cite{Di}[Corollary 4.1.23]), there is no ambiguity in using the symbol $I\chi(X)$ to denote both the Euler characteristic of the intersection homology with compact support (defined by using finite simplicial chains, like in this paper) and the Euler characteristic of the intersection homology with closed support (defined by using locally finite simplicial chains, as in (\ref{e3}) of the Introduction). Throughout this paper, we use intersection homology with compact support due to its direct connection to stratified Morse theory. Obviously, the two versions of intersection homology coincide in the compact case. 
\er

In the above notations, we have the following:
\bp\label{propEuler} 
Let $X$ be a compact complex algebraic variety of pure dimension $n$. Then the $\Qt$-modules $IH_i(X,\mathscr{L})$ are finitely generated, and 
\be\label{23}
I\chi(X)=\sum_{i=0}^{2n}(-1)^n\rank_{\Qt} IH_i(X,\mathscr{L}).
\ee
\ep

\begin{proof} 
Twisted intersection homology groups of complex algebraic varieties are finitely generated over the respective PID base ring (e.g., see \cite{Bo}[V.10.13]). In the compact case considered here, a direct proof of this fact can be given as follows.

As already noted, intersection homology groups $IH_*(X)$ can be computed simplicially, provided that a full triangulation $T$ of $X$ is chosen. Let us assume, implicitly (and without any mentioning of the upperscript $T$), that such a full triangulation $T$ has been fixed on $X$, and we lift this triangulation to a full triangulation of $X_H$. The computations in the next paragraph are with respect to these fixed full triangulations on $X$ and $X_H$, respectively. 

Since $X$ is compact, the groups $IC_i(X)$ are free abelian of finite rank. Moreover, 
by construction, $IC_i(X_H)$ is a free $\Z[H]$-module of rank $\rank_\Z IC_i(X)$, where a basis of $IC_i(X_H)$ is given by a choice of a lift of each of the allowable $i$-simplices in the above full triangulation of $X$. Hence, $IC_i(X,\mathscr{L})$ is a free $\Qt$-module of rank equal to $\rank_\Z IC_i(X)$, so the homology groups $IH_i(X,\mathscr{L})$ are finitely generated $\Qt$-modules. 
Therefore, 
\begin{equation*}
\begin{split}
I\chi(X) &=\sum_{i=0}^{2n}(-1)^n\rank_\Z IC_i(X) \\ &=\sum_{i=0}^{2n}(-1)^n\rank_{\Qt} IC_i(X,\mathscr{L}) \\ &=\sum_{i=0}^{2n}(-1)^n\rank_{\Qt} IH_i(X,\mathscr{L}).
\end{split}
\end{equation*}
\end{proof}


Let us now consider an epimorphism $\vepsi:\pi_1(X)\longrightarrow \Z$. Note that $\vepsi$ induces a rank-one local system of $\Qt$-modules $\mathscr{L}_\vepsi$ on $X$ via the following representation, which we denote by $\widetilde{\vepsi}$:
$$
\f{\widetilde{\vepsi}}{\pi_1(X)}{\Aut_{\Qt}(\Qt)\cong \Qt^*}{\alpha}{t^{\vepsi(\alpha)}}
$$
Let $X^\vepsi\longrightarrow X$ be the infinite cyclic cover induced by $\ker (\vepsi)$. We then have the following result, analogous to \cite[Theorem 2.1]{KL}. The proof is a standard check.
\bp\label{propKL}
Let $\gamma\in\pi_1(X)$ be such that $\vepsi(\gamma)=1$. Then, the map $\phi_*$ is an isomorphism of complexes of $\Qt$-modules, where
$$
\f{\phi_i}{IC_i(X,\mathscr{L}_\vepsi)}{IC_i(X^{\vepsi})\otimes \Q}{t^{-k}\otimes c}{\gamma^k\cdot c}
$$
for all $c\in IC_i(X^{\vepsi})$, $k\in\Z$. Here, $\gamma^k\cdot c$ denotes the action of $\Z$ on $IC_i(X^{\vepsi})$ by deck transformations for all $k\in\Z$, and the $\Qt$-action on $IC_i(X^{\vepsi})$ is given by $t^k\cdot c:=\gamma^k\cdot c$.
\ep

The next result follows from \cref{propEuler} and \cref{propKL}.
\bc\label{corfg}
Let $X$ be a compact complex algebraic variety of pure dimension $n$, and let $\vepsi:\pi_1(X)\longrightarrow \Z$ be an epimorphism with associated infinite cyclic cover $X^\vepsi\longrightarrow X$.
If $IH_i(X^\vepsi)$ is a finitely generated abelian group for all $i\neq n$, then $(-1)^n \cdot I\chi(X)\geq 0$. Furthermore, $I\chi(X)= 0$ if $IH_i(X^\vepsi)$ is a finitely generated abelian group for all $i$.
\ec
\begin{proof}
If $IH_i(X^\vepsi)$ is a finitely generated abelian group for all $i\neq n$, then, when regarded as a module over $\Qt$, $IH_i(X^\vepsi)$ is torsion for all $i\neq n$. Hence, by \cref{propKL},  $IH_i(X,\mathscr{L}_\vepsi)$ is also a torsion $\Qt$-module, for all $i\neq n$. Therefore, \cref{propEuler} yields that 
$$
I\chi(X)=(-1)^n \rank_{\Qt} IH_n(X,\mathscr{L}_\vepsi), 
$$
so $(-1)^n I\chi(X) \geq 0$. The second part follows similarly.
\end{proof}
\br\rm\label{remusual}
The above result remains true if we replace $IH_*$ by $H_*$ and $I\chi$ by $\chi$, since the arguments in Propositions \ref{propEuler} and \ref{propKL} are also valid in this setting, see also \cite[Section 2]{LMWa}.
\er


\section{Stratified Morse theory: overview}\label{SM}
In this section, we review the basic facts about stratified Morse theory. While the standard reference for this part is \cite{GMP}, the reader may also consult \cite{Ha,massey,MT} for a brief introduction to the theory and its applications.

Let $Z$ be a Whitney stratified complex analytic variety, analytically embedded in a smooth complex manifold $M$. Fix a Whitney stratification $\s$ of $Z$, with the usual partial order given by $S_\alpha\leq S_\beta$ for $S_\alpha, S_\beta\in \s$ such that $S_\alpha\subset\overline{S_\beta}$. 

\begin{notation}\rm For every $x\in Z$, we denote by $S_x$ the stratum of $\s$ containing $x$. \end{notation}

\begin{defn}
Let $f:Z\longrightarrow \R$ be the restriction of a smooth function $\widetilde{f}:M\longrightarrow \R$. Then 
$x\in Z$ is a {\it critical point} of $f$ if $d\widetilde{f}$ is the zero map when restricted to $T_xS_x$.
\end{defn}

\begin{defn}
For each $S_\alpha\in\s$, we denote by $T^*_{S_\alpha}M$ the {\it conormal space} to $S_\alpha$ in $M$, that is, $T^*_{S_\alpha}M$ is a vector bundle over $S_\alpha$, whose fiber $(T^*_{S_\alpha}M)_x$ over a point $x\in S_\alpha$ consists of the cotangent vectors (covectors) in $(T^*M)_x$ which vanish on the tangent space $T_x S_\alpha$ of $S_{\alpha}$ at $x$.

The set of {\it degenerate conormal covectors} to a stratum $S_\alpha$ in $\s$ is defined as
$$
D_{S_\alpha}^*M:=T^*_{S_\alpha}M\cap\bigcup_{S_\alpha<S_\beta}\overline{T^*_{S_\beta}M}=\left( \bigcup_{S_\alpha<S_\beta}\overline{T^*_{S_\beta}M} \right)_{|S_\alpha} \ ,
$$
where the closure is taken inside of $T^*M$. Thus the fiber $(D_{S_x}^*M)_x$ consists of conormal covectors to $S_x$ at $x$ which vanish on limiting tangent spaces from larger strata.
\end{defn}

\begin{defn}\label{deffunction}
A {\it Morse function} $f:Z\longrightarrow \R$ is the restriction of a smooth function $\widetilde{f}:M\longrightarrow \R$ such that:
\begin{enumerate}
\item $f$ is proper and the critical values of $f$ are distinct, i.e.\ no two critical points of $f$ correspond to the same critical value.
\item for each stratum $S$ of $Z$, $f_{|S}$ is a Morse function in the classical sense, i.e., 
 the critical points of $f_{|S}$ are nondegenerate. 
\label{itemcritical}
\item for every critical point $x$ of $f$, $(d\widetilde{f})_x\notin D_{S_x}^*M$.
\label{itemconormal}
\end{enumerate}
\end{defn}

\br\rm
Note that if $\{x\}$ is a zero-dimensional stratum of $Z$, then $x$ is a critical point of any Morse function $f:Z\longrightarrow \R$.
\er

\begin{defn}
The {\it stratified Morse index} of a Morse function $f:Z\longrightarrow \R$ at a critical point $x \in S_x$ is defined as $i_x:=s+\lambda$, where $s$ is the complex codimension of the stratum $S_x$ in $Z$, and $\lambda$ is the (classical) Morse index of $f_{|S_x}$ at $x$.
\end{defn}
 
In what follows we summarize the main results of stratified Morse theory. Let  $f:Z\longrightarrow \R$ be a Morse function. For any $c \in \R$, set $$Z_{\leq c}:=f^{-1}\left( (-\infty, c] \right).$$
\bt\label{ta}\cite[SMT Part A]{GMP} \ 
As $c$ varies within the open interval between two adjacent critical values of $f$, the topological type of $Z_{\leq c}$ does not change.
\et

Let $x \in S_x$ be a critical point of $f$, and let $V_x \subset M$ be an analytic submanifold such that $V_x$ and $S_x$ intersect transversally, and $V_x \cap S_x =\{x\}$. Let $B_{\delta}(x)$ be a closed ball of radius $\delta$ with respect to some choice of local coordinates for $M$. 
The {\it normal slice of $S_x$} (at $x$) is defined as:
\begin{equation*}N_{\delta}:=B^{\circ}_{\delta}(x) \cap V_x \cap Z.\end{equation*}
\bt\label{tb}\cite[SMT Part B]{GMP} \ 
With the above notations and assumptions, if $1 \gg \delta \gg \varepsilon >0$ are sufficiently small, then $Z_{\leq f(x)+\varepsilon}$ is homeomorphic to the space obtained from  $Z_{\leq f(x)-\varepsilon}$ by attaching 
\be\label{md} (D^{\lambda} \times D^{k-\lambda}, \partial D^{\lambda} \times D^{k-\lambda}) \times \left( (N_{\delta})_{\leq f(x)+\varepsilon}, (N_{\delta})_{\leq f(x)-\varepsilon} \right),\ee
where $k$ is the real dimension of the stratum $S_x$ containing $x$, and $\lambda$ is the (classical) Morse index of $f_{|S_x}$ at $x$.  
\et
\br\rm The expression (\ref{md}) is called the {\it Morse data at $x$}, and it is the product of the {\it tangential Morse data} (the first factor) and the {\it normal Morse data} (the second factor). Here, by product of pairs we mean $(A,B) \times  (A',B')=(A \times A', A\times B' \cup A' \times B)$.
It is also shown in \cite{GMP} that the topological types of the normal slice $N_{\delta}$ and the normal Morse data are independent of $V_x$, $\delta$ and $\varepsilon$, and, moreover, the topological type of the normal Morse data is also independent of $f$.
\er

If $\delta$ is small enough, then by choosing local coordinates on $M$, one can regard the normal slice $N_{\delta}$ as sitting directly in some affine space $\C^m$. Let $$\pi:\C^m \to \C$$ be a linear projection with $\pi(x)=0$, so that ${\rm Re}(\pi)$ restricts to a Morse function on $N_{\delta}$ (the set of such projections is open and dense in the space of all linear maps $\C^m \to \C$). 
\begin{defn}\label{cl} Let $1 \gg \delta \gg \varepsilon >0$ be  sufficiently small, let $D_{\varepsilon} \subset \C$ be a closed disk of radius $\varepsilon$, and let $\xi \in \partial D_{\varepsilon}$. The {\it complex link} of the stratum $S_x$ containing $x$ is defined by:
$$\cL_x:=\pi^{-1}(\xi) \cap \overline{N_{\delta}}.$$
The {\it cylindrical neighborhood} of $x$ is defined as:
$$C_x:=\pi^{-1}(D_{\varepsilon}) \cap \overline{N_{\delta}}.$$
\end{defn}

\br\rm
The topological types of the complex link and cylindrical neighborhood in Definition \ref{cl} are independent of $\pi$, $\delta$, $\varepsilon$, $\xi$, and of the point $x$; they depend only on the stratum $S_x$. 
\er

In the above notations, $\pi_{| \pi^{-1}(\partial D_{\varepsilon}) \cap C_x}$ is a topological fiber bundle (i.e., the  {\it Milnor fibration}), with fiber the complex link $\cL_x$. Such a fibration is determined by its monodromy homeomorphism $\h:\cL_x \to \cL_x$, which can be chosen to be the identity near $\partial \cL_x$. 
\begin{defn} The {\it variation map} $${\rm Var}_j:IH_{j-1}(\cL_x,\partial \cL_x) \to IH_{j-1}(\cL_x)$$
is defined by assigning to any $[\gamma] \in IH_{j-1}(\cL_x,\partial \cL_x)$ the element of $IH_{j-1}(\cL_x)$ determined by the chain $\gamma-\h(\gamma) \in IC_{j-1}(\cL_x)$. \end{defn}
\begin{defn} If the stratum $S_x$ containing $x$ is a singular stratum with complex link $\cL_x$, the {\it Morse group} $A_x$ is the image of the variation map
$${\rm Var}_s:IH_{s-1}(\cL_x,\partial \cL_x) \to IH_{s-1}(\cL_x).$$ 
If $x$ is a point in the nonsingular part of $Z$, the Morse group $A_x$ is defined to be  the integers, $\Z$.\end{defn}


The following result relates intersection homology with stratified Morse theory.
\bt\label{thmvariation}\cite[Part II, 6.4]{GMP}
Let $Z$ be a complex analytic variety of pure dimension (and analytically embedded in a complex manifold $M$), and let $f:Z\longrightarrow \R$ be a Morse function with a non-degenerate critical point $x$. Let $i_x=s+\lambda$ be the stratified Morse index of $f$ at $x$, 
where $s$ is the complex codimension of the stratum $S_x$ containing $x$ and $\lambda$ is the (classical) Morse index of $f_{|S_x}$ at $x$. Suppose that the interval $[a,b]$ contains no critical values of $f$ other than $f(x)\in(a,b)$.
Then:
$$IH_i(Z_{\leq b},Z_{\leq a})=
\left\{\begin{array}{lr}
A_x &  \text{if }i=i_x\\
0 &  \text{if }i\neq i_x
\end{array}\right.
$$
where $A_x$ is the Morse group associated to the stratum $S_x$ containing $x$.  
\et

\br\rm\label{clv} If we replace intersection homology by homology then, under the above assumptions, we have (cf. \cite[page 211]{GMP}):
\be\label{hrel}
H_i(Z_{\leq b},Z_{\leq a})\cong \widetilde{H}_{i-\lambda-1}(\cL_x),
\ee
where $\cL_x$ is the complex link associated to the stratum $S_x$ containing $x$. This is a consequence of Theorem \ref{tb}, together with the fact that the normal Morse data of $f$ at a critical point $x$ has the homotopy type of the pair $({\rm cone}(\cL_x),\cL_x)$, see \cite[Part II, 3.2]{GMP}. 
In particular, if $X$ is a {\it local complete intersection} of pure dimension $n$ and the interval $[a,b]$ contains no critical values of $f$ other than $f(x)\in(a,b)$, then 
\be\label{lci}
H_i(Z_{\leq b},Z_{\leq a})=0,
\ee
for all $i \neq i_x=s+\lambda$, where $s$ is the complex codimension of the stratum $S_x$ containing the critical point $x$ and $\lambda$ is the Morse index of $f_{|S_x}$ at $x$. Indeed, in this case, the complex link $\cL_x$ associated to the stratum $S_x$ has the homotopy type of a bouquet of spheres of dimension $s-1$, see \cite{Le}.
\er

\br\label{relinthom}\rm
Let us briefly comment on the meaning of the relative intersection homology groups in Theorem \ref{thmvariation}. 
In standard references, intersection homology is defined for pseudomanifolds without boundary, 
and relative intersection homology groups $IH_i(Z,U)$  are defined for pairs $(Z,U)$, with $U$ an open subspace of $Z$ (with the induced stratification). These relative intersection homology groups then fit into a long exact sequence for pairs, and respectively, triples $(Z,U,V)$  if $V \subset U$ is open.

More generally, relative intersection homology groups $IH_i(Z,U)$ can be defined if $U \subset Z$ is {\it smoothly enclosed}, in the sense that if $Z$ is analytically embedded in some ambient complex analytic manifold $M$ then $U=Z \cap M'$, where $M'$ is a complex analytic submanifold of $M$ of $\dim M'=\dim M$, whose boundary $\partial M'$ intersects each stratum of $Z$ transversally. Furthermore, if $V \subset U$ is smoothly enclosed, then the resulting relative intersection homology groups fit into long exact sequences for pairs and, respectively, triples. For more details, see \cite[Section 1.3.2]{MP}.

Now, if $\widetilde{f}:M\to \mathbb{R}$ restricts to a Morse function $f$ on $Z$, and $a \in \mathbb{R}$ is not a critical value for $f$, then $Z_{\leq a}=Z \cap \widetilde{f}^{-1}\left((-\infty,a]\right)$ is smoothly enclosed in $Z$ since the boundary $\widetilde{f}^{-1}(a)$ of $\widetilde{f}^{-1}\left((-\infty,a]\right)$ is the preimage of a regular value, hence a submanifold of $M$ transverse to each stratum of $Z$.
In particular, one can make sense of the relative intersection homology groups  appearing in the statement of Theorem \ref{thmvariation}.

Alternatively, if $(X,\partial X)$ is a pseudomanifold with collared boundary $\partial X$, one can define $IH_i(X)$ as in \cite{GM}, by replacing $X$ by $X^{\circ}:=X \setminus \partial X$. Indeed, if $T$ is an open collared neighborhood of $\partial X$, then $X^{\circ}$ is a pseudomanifold, and $T^{\circ}:=T \setminus \partial X$ is an open subspace. Furthermore, one can construct a stratum-preserving homotopy equivalence of pairs $(X,\partial X) \sim (X^{\circ}, T^{\circ})$, therefore $IH_i(X)\cong IH_i(X^{\circ})$ and  $IH_i(X,\partial X)\cong IH_i(X^{\circ}, T^{\circ})$. 
\er

\medskip

The following circle-valued version of a Morse function will be needed in the sequel.

Let $\widetilde{f}:M\longrightarrow S^1$ be a smooth function. We can identify $S^1$ with $\R/\Z$ and consider $\widetilde{f}$ as a multi-valued real function. With this consideration in mind, we can define circle-valued Morse functions as follows:
\begin{defn}
A {\it circle-valued Morse function} $f:Z\longrightarrow S^1$ is the restriction of a smooth function $\widetilde{f}:M\longrightarrow S^1$ such that:
\begin{enumerate}
\item $f$ is proper and the critical values of $f$ are distinct, i.e., no two critical points of $f$ correspond to the same critical value (modulo $\Z$).
\item for each stratum $S$ of $Z$, the critical points of $f_{|S}$ are nondegenerate (i.e., if $\dim(S)\geq 1$, the Hessian matrix of $f_{|S}$ is nonsingular).
\label{itemcritical}
\item for every critical point $x$ of $f$, $(d\widetilde{f})_x\notin D_{S_x}^*M$.
\label{itemconormal}
\end{enumerate}
\label{defcircle}
\end{defn}


\section{Existence of circle-valued Morse functions}\label{Ex}
Let $X$ be a pure $n$-dimensional closed subvariety of an abelian variety $G$ of complex dimension $N$. 
Let $\Gamma$ be the space of left invariant holomorphic $1$-forms on $G$.  Then $\Gamma$ is a complex affine space of  dimension $N$. Moreover, since $G$ is a complex abelian Lie group, every left invariant form is closed. 

\begin{defn}
We say that a real $1$-form $\mu$ on $G$ has a {\it singularity} at $x\in X$ if any smooth function $f:V_x\longrightarrow \R$ defined on a  neighborhood $V_x\subset S_x$ of $x$ such that $df=\mu_{|V_x}$, has a critical point at $x$. We say that $\mu$ has a {\it regular singularity} at $x\in X$ if any $f$ as above has a nondegenerate critical point at $x$. (Here, $S_x$ denotes as before the stratum of $X$ containing the point $x$.)
\end{defn}

\bl\label{lemzariski1}
There exists a Zariski open set $U\subset\Gamma$ such that for every $\eta\in U$, there are a finite number of singularities of $\re(\eta)$ in $X$.
\el

\begin{proof}
We follow here the notations from \cite[Section 6]{LMWa}. For each stratum $S\in \s$, we define the degenerating locus in $S\times\Gamma$ as
$$
Z_S:=\{(x,\eta)\in S\times\Gamma\ |\ \eta_{|S} \text{ vanishes at }x \}
$$
Since $S$ is locally closed in the Zariski topology of $G$, and the degenerating condition is given by algebraic equations, $Z_S$ is an algebraic set of $S\times\Gamma$. Since $S$ is smooth, the first projection $p_1:Z_S\longrightarrow S$ is a complex vector bundle whose fiber has (complex) dimension $N-\dim_{\C}(S)$. Hence $Z_S$ is  smooth manifold and $\dim_\C Z_S=N$.

Consider now the second projection $p_2:Z_S\longrightarrow \Gamma$. By the Cauchy-Riemann equations, and the fact that the elements of $\Gamma$ are holomorphic $1$-forms, the points $(x,\eta)\in p_2^{-1}(\eta)$ correspond to the points $x\in S$ such that $\re(\eta)_{|S}$ has a singularity at $x$. By Verdier's generic fibration theorem \cite[Corollary 5.1]{Ve} (and using the fact that $\dim_{\C} Z_S=\dim_{\C} \Gamma$), 
there is a Zariski open subset $U_S$ of $\Gamma$ such that $p_2^{-1}(\eta)$ consists of a finite (or 0) number of points for all $\eta$ in $U_S$. Since there are a finite number of strata in $X$, the open set $U:=\bigcap_{S\in\s}U_S$ satisfies the properties required in the statement of the Lemma.
\end{proof}

Let $U$ be as in \cref{lemzariski1}, and let $Z_S$ and $p_2:Z_S\longrightarrow \Gamma$ be defined as in the proof of \cref{lemzariski1}, that is,
$$
Z_S:=\{(x,\eta)\in S\times\Gamma\ |\ \eta_{|S} \text{ vanishes at }x \}
$$
for all $S\in \s$, and $p_2:Z_S\longrightarrow \Gamma$ is the projection. Recall that $Z_S$ is smooth and $\dim_\C Z_S=N$. Let $\eta\in U$. As discussed in the proof of \cref{lemzariski1}, we have that
$$
p_2^{-1}(\eta)=\{(x,\eta)\ |\ \re(\eta)_{|S} \text{ has a singularity at } x\}.
$$

\bl\label{lemtrans}
There exists a Zariski open set $U'_S\subset U$ such that for all $\eta\in U'_S$, $Z_S$ intersects $S\times\{\eta\}$ transversally at all points of intersection.
\el
\begin{proof}
We use generic smoothness in the target (e.g., see \cite[Theorem 25.3.3]{ravi}) to see that there exists a Zariski open subset $U'_S\subset U$ such that $(p_2)_{|p_2^{-1}(U'_S)}:p_2^{-1}(U'_S)\longrightarrow U'_S$ is a smooth morphism. Note that $p_2^{-1}(U'_S)$ could be empty, in which case for any $\eta\in U'_S$, $Z_S$ does not intersect $S\times\{\eta\}$, so the lemma would then be true trivially.

Assume now that $p_2^{-1}(U'_S)$ is not empty. In this case, $p_2:p_2^{-1}(U'_S)\longrightarrow U'_S$ is a surjective submersion. In the proof of \cref{lemzariski1}, we showed that $p_2:p_2^{-1}(U)\longrightarrow U$ has finite fibers, so
$$
p_2:p_2^{-1}(U'_S)\longrightarrow U'_S
$$
is a finite covering map.

Let $\eta\in U'_S$, and let $(x,\eta)\in p_2^{-1}(\eta)=Z_S\cap \left(S\times\{\eta\}\right)$. Since $Z_S$ and $S\times\{\eta\}$ have complimentary dimensions in $S\times\Gamma$, they intersect transversally at $(x,\eta)$ if and only if
\be\label{tr}
T_{(x,\eta)}Z_S\cap T_{(x,\eta)}\left(S\times\{\eta\}\right)=0,
\ee
where both tangent spaces are seen in $T_{(x,\eta)}\left(S\times\Gamma\right)$. We have the following commutative triangle
$$
\begin{tikzcd}
Z_S \arrow[r,hook,"i"]\arrow[d,"p_2"] & S\times\Gamma \arrow[ld,"p_2'"]\\
\Gamma & \\
\end{tikzcd}
$$
where $p_2':S\times\Gamma\longrightarrow \Gamma$ is the projection. Let $v\in T_{(x,\eta)}Z_S\cap T_{(x,\eta)}\left(S\times\{\eta\}\right)$. 
Since $v\in T_{(x,\eta)}\left(S\times\{\eta\}\right)$, we have that
$$
(dp_2')_{(x,\eta)}(v)=0.
$$
Since $v \in  T_{(x,\eta)}Z_S$, by the commutativity of the above triangle, we get that 
\be\label{tr2}
(dp_2)_{(x,\eta)}(v)=0.
\ee
Furthermore, since $\eta\in U'_S$, $(dp_2)_{(x,\eta)}$ is an isomorphism. So (\ref{tr2}) implies that $v=0$. This proves (\ref{tr}) and the Lemma.
\end{proof}

\bl\label{lemzariski2}
There exists a Zariski open set $U'$ in $\Gamma$ such that for every $\eta \in U'$, all the singularities of $\re(\eta)$ in $X$ are regular singularities, and there is a finite number of them.
\el
\begin{proof}
Consider the Zariski open set  $$U':=\bigcap_{S\in \s}U'_S,$$ where $U'_S$ are defined as in \cref{lemtrans}. Note that $U'\subset U$, where $U$ is as in \cref{lemzariski1}, so for every $\eta\in U'$, $\re(\eta)$ has a finite number of singularities in $X$. The only thing left to prove is that the $1$-form $\re(\eta)_{|S_x}$ has a regular singularity at $x\in X$ if $Z_{S_x}$ intersects $S_x\times\{\eta\}$ transversally at $(x,\eta)$, and then apply \cref{lemtrans}. The rest of the proof is devoted to show this fact.

After a change of (complex) coordinates, we can assume that locally, near the point $x\in S$, there is a system of coordinates $(z_1,\ldots,z_N)$ of $G$ such that $S$ is given by the equations $z_1=\ldots=z_{N-k}=0$, where $k=\dim_\C S$. Let $l_1,\ldots,l_N$ be a basis of left-invariant holomorphic $1$-forms (i.e., a basis of $\Gamma$), and  endow $\Gamma$ with the coordinates $(\alpha_1,\ldots,\alpha_N)$ corresponding to the fixed basis $l_1,\ldots,l_N$.

We consider both $Z_S$ and $S\times\{\eta\}$ as subspaces of $G\times\Gamma$, and we identify $T_{(x,\eta)}\left(G\times\Gamma\right)$ with $T_{x}G\times T_{\eta}\Gamma$. We consider $T_{x}G$ and $T_{\eta}\Gamma$ as vector spaces with the natural fixed bases coming from the $z_i$'s and $\alpha_i$'s, respectively.

We can locally express the $l_i$'s in the basis $(dz_1,\ldots,dz_N)$ as
$$
l_i=\sum_{j=1}^N \beta_{ij}(z_1,\ldots,z_N)dz_j,
$$
for some holomorphic functions $\beta_{ij}$, for $i,j=1,\ldots,N$.

Note that
$$
T_{(x,\eta)}\left(S\times\{\eta\}\right)=\{(v,0)\in T_{x}G\times T_{\eta}\Gamma\text{ }|\text{ }\text{the first }N-k \text{ coordinates of }v\text{ are }0\}.
$$
Furthermore, the equations that cut out $Z_S$ near $(x,\eta)$ are
$$
\begin{array}{ccccc}
z_1=0 \ ,  \ldots  , \ z_{N-k}=0,
\end{array}
$$
and
$$
\left(\sum_{i=1}^N \alpha_i l_i(0,\ldots,0,z_{N-k+1},\ldots,z_N)\right)_{|S}=0.
$$
This is equivalent to the following system of equations in $z_1,\ldots,z_N,\alpha_1,\ldots,\alpha_N$:
$$
\left\{\begin{array}{lll}
z_i=0 & \text{\ \ \ for }i=1,\ldots,N-k,\\
\sum_{i=1}^N \alpha_i\beta_{ij}(0,\ldots,0,z_{N-k+1},\ldots,z_N)=0 & \text{\ \ \ for }j=N-k+1,\ldots, N.\\
\end{array}\right.
$$
The tangent space $T_{(x,\eta)}Z_S$ is given by
$$
T_{(x,\eta)}Z_S=\left\{(v,w)\in T_{x}G\times T_{\eta}\Gamma\mid (df)_{(x,\eta)}(v,w)=0 \text{ for all }f\text{ cutting out }Z_S\right\},
$$
so, if $\eta=\sum_{i=1}^N a_il_i$, we have that
\begin{multline*}
T_{(x,\eta)}Z_S=\\ \left\{(0,\ldots,0,v_{N-k+1},\ldots,v_N,w_1,\ldots,w_N)\mid\begin{array}{c}\sum\limits_{r=N-k+1}^N\left(\sum\limits_{i=1}^N a_i\frac{\partial \beta_{ij}}{\partial z_r}(x)\right)v_r+\sum\limits_{i=1}^N\beta_{ij}(x)w_i=0, \\j=N-k+1,\ldots,N\end{array}\right\}.
\end{multline*}
By the above descriptions of $T_{(x,\eta)}Z_S$ and $T_{(x,\eta)}\left(S\times\{\eta\}\right)$, we get that
\begin{multline*}
T_{(x,\eta)}Z_S\cap T_{(x,\eta)}\left(S\times\{\eta\}\right)=\\ \left\{(0,\ldots,0,v_{N-k+1},\ldots,v_N,0,\ldots,0)\mid\begin{array}{c}\sum\limits_{r=N-k+1}^N\left(\sum\limits_{i=1}^N a_i\frac{\partial \beta_{ij}}{\partial z_r}(x)\right)v_r=0\\j=N-k+1,\ldots,N\end{array} \right\}.
\end{multline*}
Let $V_x\subset S$ be a neighborhood of $x$ in $S$. If $f:V_x\longrightarrow \C$ is a holomorphic function such that $df=\eta_{|{V_x}}$, then 
the Hessian matrix of $f$ at the point $x$ is exactly the matrix $H$ with entries
$$
h_{jr}=\sum\limits_{i=1}^N a_i\frac{\partial \beta_{ij}}{\partial z_r}(x)
$$
for $i,j=N-k+1,\ldots,N$. Hence, $x$ is a non-degenerate critical point of $f$ if and only if there is no non-zero $(v_{N-k+1},\ldots,v_N)$ such that
$$
H\cdot\left(\begin{array}{c}
v_{N-k+1}\\
\vdots\\
v_N
\end{array}\right)=\left(\begin{array}{c}
0\\
\vdots\\
0\\
\end{array}\right)
$$
This happens if and only if $T_{(x,\eta)}Z_S\cap T_{(x,\eta)}\left(S\times\{\eta\}\right)$ is only the $0$ vector, which is true by \cref{lemtrans}.

We have shown that, if $\eta\in U'$, then for every singularity $x$ of $\re(\eta)$ and for every holomorphic function $f$ defined locally near $x \in S$ so that $df=\eta$ in a neighborhood of $x$,
$x$ is a non-degenerate critical point of $f$. Hence, by the holomorphic Morse lemma, $x$ is a non-degenerate critical point of $\re(f)$ of Morse index $\dim_\C(S)$. Since $d(\re(f))=\re(\eta)$ in a neighborhood of $x$, we deduce that $x$ is a regular singularity of $\re(\eta)$. 
\end{proof}

\bl\label{lemzariski3}
There exists a Zariski open set $U''$ in $\Gamma$ such that for every $\eta\in U''$, $\re(\eta)_x\notin D^*_{S_x}G$ for any $x\in X$.
\el
\begin{proof}
Let $S_\alpha,S_\beta\in\s$ such that $S_\alpha\subset \overline{S_\beta}$. We need to show that there exists a Zariski open subset $U_{\alpha,\beta}$ of $\Gamma$ such that, if $\eta\in U_{\alpha,\beta}$, then
$$
\re(\eta)_x\notin T^*_{S_\alpha}G\cap\overline{T^*_{S_\beta}G}
$$
for any $x\in S_\alpha$. Then the Zariski open set $U''$ we are looking for will be defined as the intersection of all of these $U_{\alpha,\beta}$ corresponding to pairs of strata $S_\alpha,S_\beta\in\s$ such that $S_\alpha\subset \overline{S_\beta}$ (there are finitely many such pairs).

Since $G$ is a complex Lie group, its cotangent bundle is trivial, and we can identify $G\times\Gamma$ with $T^*G$. If we see $S\times\Gamma$ as a locally closed Zariski set in $G\times \Gamma$, we can identify $Z_S$ (as defined in the proof of \cref{lemzariski1}) with $T^*_S G$ for any stratum $S\in\s$.

With the above identifications, we have that $T^*_{S_\alpha}G\cap\overline{T^*_{S_\beta}G}=Z_{S_\alpha}\cap\overline{Z_{S_\beta}}$, where the closure is taken with respect to the Zariski topology. Notice that, as discussed in the proof of \cref{lemzariski1}, $\dim_\C\left(Z_{S_\alpha}\right)=\dim_\C\left(Z_{S_\beta}\right)=N$. Also, $Z_{S_\beta}$ is a smooth variety, so it is irreducible, and thus its Zariski closure $\overline{Z_{S_\beta}}$ in $T^*G$ is an irreducible $N$-dimensional variety as well. Hence, $\overline{Z_{S_\beta}}\backslash Z_{S_\beta}$ is a variety of dimension
\begin{equation}
\dim_\C\left(\overline{Z_{S_\beta}}\backslash Z_{S_\beta}\right)\leq N-1.
\label{eqndim}
\end{equation}
This also shows that the Zariski closure of $Z_{S_\beta}$ is the same as the closure in the usual complex topology, because an analytic neighborhood of any point in $\overline{Z_{S_\beta}}\backslash Z_{S_\beta}$ (where the closure is in the Zariski topology) cannot be contained in an ($N-1$)-dimensional space. From now on, we will not make any distinction between the closures in the two topologies.

Notice also that $Z_{S_\alpha}\cap\overline{Z_{S_\beta}}\subset \overline{Z_{S_\beta}}\backslash Z_{S_\beta}$, so $Z_{S_\alpha}\cap\overline{Z_{S_\beta}}$ is a constructible set (i.e., a finite union of locally closed sets) of dimension at most $N-1$.

In the proof of \cref{lemtrans}, we showed that there exists a Zariski open subset $U'_{S_\alpha}$ such that the morphism of varieties given by the projection onto the second coordinate
$$
p_2:Z_{S_\alpha}\cap p_2^{-1}(U'_{S_\alpha})\longrightarrow U'_{S_\alpha}
$$
is a finite covering map, so in particular it is \' etale, and thus finitely presented. By Chevalley's theorem on constructible sets (\cite[7.4.2]{ravi}), $p_2(Z_{S_\alpha}\cap\overline{Z_{S_\beta}}\cap p_2^{-1}(U'_{S_\alpha}))$ is a constructible subset (in the Zariski topology) of $U'_{S_\alpha}$. By equation (\ref{eqndim}), this image cannot contain a Zariski open subset of $U'_{S_\alpha}$. Thus, $p_2(Z_{S_\alpha}\cap\overline{Z_{S_\beta}}\cap p_2^{-1}(U'_{S_\alpha}))$ is contained in some proper closed set $C$.

Let $U_{\alpha,\beta}:=U'_{S_\alpha}\backslash C$. We have proved that if $\eta\in U_{\alpha,\beta}$, then $\eta_x\notin T^*_{S_\alpha}G\cap\overline{T^*_{S_\beta}G}$ for any $x\in S_\alpha$. To conclude the proof, it remains to see that the latter fact implies that $\re(\eta)_x\notin T^*_{S_\alpha}G\cap\overline{T^*_{S_\beta}G}$, where the bundles are now seen as bundles of real manifolds. 

We have the following globally defined maps, whose composition is an isomorphism of real bundles:
$$
\begin{array}{ccccc}
(T^* G)_{\R} & \longrightarrow & (T^* G)_{\R}\otimes_\R \C\cong (T^* G)_{\R}\oplus i (T^* G)_{\R}& \longrightarrow & (T^* G)_{\C}\\
\end{array}
$$
where the first map is given by the inclusion into the first summand, and the second map is given by the projection onto the holomorphic part. The subscripts $\R$ and $\C$ are added to distinguish between real and holomorphic bundles. 
Let $z_j=x_j+iy_j$ be holomorphic coordinates in $G$ near a point $p\in X$, for $j=1,\ldots, N$, such that the equations cutting out $S_p$ near $p$ are $z_1=\ldots=z_N-k=0$, where $k=\dim_\C(S_p)$. Locally, the above composition takes the form
$$
\begin{array}{ccccc}
(T^* G)_{\R} & \longrightarrow & (T^* G)_{\R}\otimes_\R \C\cong (T^* G)_{\R}\oplus i (T^* G)_{\R}& \longrightarrow & (T^* G)_{\C}\\
dx_j&\mapsto &dx_j=\frac{dz_j+d\overline{z_j}}{2}& & \\
dy_j&\mapsto &dy_j=\frac{dz_j-d\overline{z_j}}{2i}& & \\
& & dz_j&\mapsto& dz_j\\
& & d\overline{z_j}&\mapsto& 0\\
\end{array}
$$
Hence, this isomorphism of (real) bundles takes $(T_{S_p}^* G)_\R$ to $(T_{S_p}^* G)_\C$ for all $p$ in $X$. Thus, this isomorphism takes $(T^*_{S_\alpha}G\cap\overline{T^*_{S_\beta}G})_\R$ to $(T^*_{S_\alpha}G\cap\overline{T^*_{S_\beta}G})_\C$, and $\re(\eta)_x$ to $\frac{1}{2}\eta_x$. So, if $\eta\in U_{\alpha,\beta}$, then $$\re(\eta)_x\notin (T^*_{S_\alpha}G\cap\overline{T^*_{S_\beta}G})_\R$$ for any $x \in S_\alpha$.
\end{proof}

The proof of the next lemma follows from standard arguments in classical Morse theory using step functions (e.g., see \cite[Lemma 2.8]{cobordism}), which can be translated to the stratified setting by using Whitney's condition (A). We include the details for the convenience of the reader.

\bl\label{lemperturb}
Let $\widetilde{f}:M \to S^1$ be a smooth map, so that $f=\widetilde{f}_{|Z}:Z\longrightarrow S^1$ is a circle-valued function satisfying all of the conditions in \cref{defcircle} except for maybe having distinct critical values. Suppose $f$ has a finite number of critical points. Then, we can perturb $\widetilde{f}$ to find a function $\widetilde{h}:M\longrightarrow S^1$ such that $h:=\widetilde{h}_{|Z}:Z \to S^1$ satisfies the following properties:
\begin{itemize}
\item $h$ agrees with $f$ outside of a small neighborhood of the critical points.
\item the set of critical points of $f$ and $h$ coincide.
\item $h:Z\longrightarrow S^1$ is a circle-valued Morse function (in particular, the critical values of $h$ are all distinct).
\item the stratified Morse index of $f$ at a critical point is the same as that of $h$.
\end{itemize}
\el
\begin{proof}
Suppose that $f(p_1)=f(p_2)$ for two distinct critical points $p_1$ and $p_2$ of $f$. 

Let $\lambda:M\longrightarrow[0,1]$ be a smooth function such that
\begin{itemize}
\item $\lambda$ is identically $1$ in a small neighborhood $U$ of $p_2$.
\item $\lambda$ is identically $0$ outside of a larger neighborhood $V$ of $p_2$, with $U\subset V$.
\item $\overline V$ is compact.
\item There are no critical points of $f$ in $\overline V$ other than $p_2$. 
\end{itemize}

Let $K=\overline V\backslash U$, which is a compact set, and fix a Riemannian metric $||\cdot||_M$ on $M$, which induces metrics $||\cdot||_S$ on every stratum $S$ of the fixed Whitney stratification $\s$ of $Z$. We first show that there exist positive constants $c,c'$ such that
\begin{equation}
0<c\leq ||\nabla_x f_{|S_x}||_{S_x} \ \ \ {\rm and} \ \ \ 
||\nabla_x \lambda_{|S_x}||_{S_x}\leq c'
\label{eqnbounds}
\end{equation}
for every $x\in Z\cap K$, where $\nabla_x f_{|S_x}$ denotes the gradient of $f_{|S_x}$ at the point $x$ (recall that we are considering $f$ as a multi-valued real function). The existence of these constants is trivial in the non-stratified case discussed in \cite[Lemma 2.8]{cobordism} given the fact that $K$ is compact, but in our case $S_x\cap K$ is not necessarily compact, so more care is needed.

By compactness, we can find $c'>0$ such that $||\nabla_x \lambda||_M \leq c'$ for all $x\in Z\cap K$. Now, we  just have to use that
$$
{||\nabla_x \lambda_{|S_x}||}_{S_x}\leq ||\nabla_x \lambda||_M \leq c'  ,
$$
which gives the constant $c'$. To find $c$, we argue by contradiction. Suppose that there exists a subsequence of points $\{x_i\}\in Z\cap K$ such that
$$\lim_{i\to\infty}||\nabla_{x_i} f_{|S_{x_i}}||_{S_{x_i}}=0.$$
By passing to a subsequence, we can assume that
\begin{itemize}
\item $\{x_i\}\subset S_\beta\cap K$ for some $S_\beta\in\s$.
\item $\lim_{i\to\infty}x_i=x$ for some $x\in S_\alpha\cap K$, where $S_\alpha\in\s$ satisfies that $S_\alpha\subset \overline{S_\beta}$.
\item The tangent planes $T_{x_i}S_\beta$ converge to a $(\dim_\R(S_\beta))$-plane $T$ as $i \to \infty$.
\end{itemize}
Since $\s$ is a Whitney stratification, Whitney's condition (A) yields that
$$
T_x S_\alpha\subset T=\lim_{i\to \infty} T_{x_i}S_\beta  .
$$
Let $v\in T_x S_\alpha$ with $||v||_{S_\alpha}=1$. Since $T_x S_\alpha\subset T$, there exist $v_i\in T_{x_i} S_\beta$ such that $\lim_{i\to\infty}v_i=v$, and we can take these $v_i$ such that $||v_i||_{S_\beta}=1$ for all $i$. We have that
$$
df_{|S_\alpha}(v)=df(v)=\lim_{i\to \infty} df(v_i)=\lim_{i\to \infty} df_{|S_\beta}(v_i)=\lim_{i\to \infty}\langle \nabla_{x_i} f_{|S_{\beta}},v_i\rangle_{S_\beta}
$$
and
$$
|\lim_{i\to \infty}\langle \nabla_{x_i} f_{|S_{\beta}},v_i\rangle_{S_\beta}|=\lim_{i\to \infty}|\langle \nabla_{x_i} f_{|S_{\beta}},v_i\rangle_{S_\beta}|\leq \lim_{i\to\infty}||\nabla_{x_i} f_{|S_{\beta}}||_{S_{\beta}}=0,
$$
so
$$
df_{|S_\alpha}(v)=0.
$$
Since this last equality is true for all $v\in T_x S_\alpha$ with $||v||_{S_\alpha}=1$, we get that $f_{|S_\alpha}$ has a critical point at $x$, which contradicts our assumption that $f$ has no critical points on $K$. This concludes the proof of the existence of positive constants $c,c'$ as in (\ref{eqnbounds}).

Let $\widetilde g=\widetilde f+\vepsi\lambda:M\longrightarrow S^1$ for some $\vepsi>0$, and let $g:Z\longrightarrow S^1$ be its restriction to $Z$. We have that, for all $x\in Z\cap K$,
$$
||\nabla_x(f+\vepsi\lambda)_{|S_x}||_{S_x}\geq ||\nabla_x f_{|S_x}||_{S_x}-\vepsi||\nabla_x \lambda_{|S_x}||_{S_x}\geq c-\vepsi c' , 
$$
so, for $\vepsi$ small enough (smaller than $\frac{c}{c'}$),  $g$ does not have any critical values on $Z\cap K$. Note that, for $\vepsi$ small enough, $g$ satisfies all of the conditions in \cref{defcircle} except for maybe having distinct critical values. Moreover, we can pick $\vepsi$ small enough so that $g(p_2)\neq g(p)$ for any $p\neq p_2$ critical point of $f$ (or $g$). 

The desired function $h$ is constructed after a finite number of such steps, since $f$ only has a finite number of critical points.
\end{proof}

\br\label{remhe}\rm
Using the notation in the proof of \cref{lemperturb}, note that $\widetilde g$ is homotopy equivalent to $\widetilde f$ by the homotopy
$$
\f{H}{M\times[0,1]}{S^1}{(x,t)}{f(x)+t\vepsi\lambda(x)}
$$
Hence, in the same notation as \cref{lemperturb}, the circle Morse function $h:Z\longrightarrow S^1$ is homotopy equivalent to $f$.
\er

We can now prove the main result of this section, on the existence of circle-valued Morse functions on closed subvarieties of an abelian variety. 

For this discussion, assume that $X$ is connected and contains the identity element $e$ of $G$, and let $i: X \hookrightarrow G$ denote the inclusion.
Let $\eta \in \Gamma$ be a left-invariant holomorphic $1$-form on $G$ with integral real part, i.e., \ with $\re(\eta)$ belonging to the image of the natural map:
$$H^1(G;\Z) \rightarrow H^1(G;\R).$$
(As we will see below, the set of such $\eta$ is Zariski dense in $\Gamma$.) 
Let $\Lambda \subset \Z$ denote the image of $i_*H_1(X;\Z)$ under $\Re(\eta): H_1(G; \R) \rightarrow \R$; $\Lambda$ is either $0$ or a lattice in $\R$.
We then obtain a well-defined integration map:
\be\label{ph}
\f{\phi_\eta}{X}{\R/\Lambda}{x}{\int_{e}^{x}\re(\eta)}
\ee
where the integral is taken over any path in $X$ from $e$ to $x$. 
We gather some properties of $\phi_\eta$:
\begin{enumerate}[(i)]
\item if $T$ is an open regular neighborhood of $X$ in $G$ (see \cite{Du} for definition and existence), then the fact that the inclusion $X \hookrightarrow T$ is a homotopy equivalence implies that $\phi_\eta$ naturally extends to a smooth map $T \rightarrow \R/\Lambda$; furthermore the differential of this extension at any $x \in X$ is equal to $\Re(\eta)_x$, after we have made the natural identification $T_{\phi_\eta(x)}(\R/\Lambda) = \R$.
\item the induced map $(\phi_\eta)_*: H_1(X;\Z) \rightarrow H_1(\R/\Lambda;\Z) = \Lambda$ is exactly equal to the surjection $\Re(\eta) \circ i_*: H_1(X;\Z) \twoheadrightarrow \Lambda$.     
\end{enumerate}
Property (i) provides the essential information to decide whether $\phi_\eta$ is stratified Morse and property (ii) will simplify several arguments in the next section.
The next result shows that we can pick $\eta$ so that a small perturbation of the function $\phi_\eta$ defined by (\ref{ph}) is a stratified Morse function. 
\begin{theorem}\label{thmexistence}
If $X$ is a connected, pure $n$-dimensional subvariety of an abelian variety $G$, there exist circle-valued Morse functions $f : X \longrightarrow S^1$ such that:
\begin{itemize}
\item $f$ can be extended to a regular neighborhood of $X$ in $G$.
\item the induced map $f_*: H_1(X;\Z) \rightarrow H_1(S^1;\Z)$ is surjective if $n > 0$.
\item the stratified Morse index of $f$ at every critical point is $n$.
\end{itemize}  
\end{theorem}
\begin{proof}
By translation in $G$, we ensure that $e \in X$.
The statement is trivial in dimension zero, so we assume $n > 0$.
Let $U'$ be the Zariski open set described in \cref{lemzariski2}, $U''$ be the Zariski open set from \cref{lemzariski3} and set $W=U'\cap U''$. Let
$$
R=\{\eta\in\Gamma\text{\ }\mid\text{\ }\re(\eta) \text{ is integral}\}.
$$
By using \cite[Proposition 4.6]{LMWa}, it can be seen that $R$ is Zariski dense in $\Gamma$. Therefore, $R\cap W\neq\emptyset$.
Pick an $\eta$ belonging to this set $R\cap W$.
Let $\phi_\eta$, $\Lambda$, and $T$ be as in (\ref{ph}).

We first show that $\phi_\eta$ is circle-valued. 
This is true only if $\Lambda \neq 0$.
Assume towards a contradiction that $\Lambda = 0$.
Then $\phi_\eta: X \rightarrow \R$ is locally the real part of a holomorphic function.
By property (i) of (\ref{ph}) the differential of $\phi_\eta$ equals $\re(\eta)$ along $X$.
Since $\re(\eta)$ has finitely many singular points on $X$, and $X$ has pure dimension $> 0$, it follows that $\phi_\eta$ is nonconstant on every irreducible germ of $X$.
By local considerations and \cite[Theorem II.5.7]{De} our map $\phi_\eta$ must be open. 
This is impossible, because $X$ is compact.  Therefore,  $\phi_\eta$ is circle-valued.

Since $X$ is compact, $\phi_\eta$ is proper.
By \cref{lemzariski2}, \cref{lemzariski3} and property (i) of (\ref{ph}), $\phi_\eta$ satisfies all of the conditions in \cref{defcircle} except for maybe having distinct critical values. 
By \cref{lemperturb}, we can perturb $\phi_\eta$ to get a circle-valued Morse function that we call $f$, extendible to $T$.
By \cref{remhe} and property (ii) of (\ref{ph}), $f_*$ induces a surjection on first homology.

Recall that in the proof of \cref{lemzariski2}, we saw (using the holomorphic Morse lemma) that the Morse index of $\left(\phi_\eta\right)_{|S}$ at a critical point $x\in S$ is $\dim_\C S$, where $S\in \s$. Hence, the stratified Morse index of $f$ at every critical point $x$ is $\dim_\C(S_x)+\codim_\C(S_x,X)=\dim_\C(X)=n$.
\end{proof}

\br\rm \label{ana}
We note here that all constructions in this section can also be performed (with only minor modifications of the arguments involved) in the analytic setting, and all results still hold true for connected pure-dimensional subvarieties of complex compact tori $\mathbb{C}^N/{\mathbb{Z}^{2N}}$. In order to see this, we consider the compactification $\overline{T^*G}:=\mathbb{P}(T^*G \oplus \mathbb{1}_G) \cong G \times \mathbb{C}\mathbb{P}^N$ of the cotangent bundle $T^*G=G \times \Gamma$, obtained by fiberwise projectivization (here $\mathbb{1}_G$ denotes the trivial complex line bundle on $G$), together with the analytic closures  $\overline{T^*_SG}$ (in $\overline{T^*G}$) of the conormal bundles $T^*_SG=Z_S$ of strata in $X$. The set $\overline{T^*_SG}|_S$ is a  $\mathbb{C}\mathbb{P}^{N-\dim(S)}$-bundle on $S$, and the projection $p_2:T^*_SG \to \Gamma\cong \mathbb{C}^N$ can be regarded as the restriction over an open set of a proper holomorphic (hence stratifiable) map  $\overline{p_2}:\overline{T^*_SG} \to \mathbb{C}\mathbb{P}^N$. Since strata in the target of $\overline{p_2}$ are subsets of a projective space, they are in fact algebraic (by Chow's Lemma). So all our arguments (e.g., Verdier's generic fibration theorem, Chevalley's theorem on constructible sets, generic smoothness) still apply in the context of the proper holomorphic map $\overline{p_2}$, with the generic open set $W$ of (the proof of) Theorem \ref{thmexistence} still a Zariski open in the algebraic sense.
\er

\section{Proof of the signed Euler characteristic property}\label{PR}
In this section we prove the signed Euler characteristic identities (\ref{e2}) and (\ref{e3}) from the Introduction in the case of closed subvarieties of abelian varieties.

We begin by recalling Thom's first isotopy lemma, e.g., see \cite[Chapter 1, 3.5]{dimca}.
\bt[Thom's first isotopy lemma]
Let $Z$ be a locally closed subset in the smooth manifold $M$, and let $\s$ be a Whitney stratification of $Z$. Let $f:M\longrightarrow N$ be a smooth map of manifolds such that $f_{|S}:S\longrightarrow N$ is a submersion and $f_{|\overline{S}}$ is proper for any stratum $S\in\s$. Then, $f$ fibers $Z$ over $N$ topologically locally trivially in the stratified sense, i.e., the trivialization homeomorphism $$\varphi:(f^{-1}(p)\cap Z)\times U\longrightarrow f^{-1}(U)
\cap Z$$
preserves the strata of the obvious stratifications induced by $\s$ on both the source and the target of $\varphi$.
\label{thmthom}
\et

The following lemma will be needed in the proof of our main result.
\bl\label{lemfiniteIH}
Let $f:X\longrightarrow S^1$ be a circle-valued Morse function as in \cref{thmexistence}. We identify $S^1$ with $\R/\Z$, and let $a,b\in \R$ with $a<b$ and $b-a<1$. Suppose there are no critical values of $f$ in $(a,b)$. Then the intersection homology groups
$$
IH_i\left(f^{-1}\big((a,b)\big)\right)
$$
are finitely generated abelian groups for all $i$.
\el
\begin{proof}
Since there are no critical values of $f$ in $(a,b)$, and $S^1$ is one-dimensional, we can apply Thom's first isotopy lemma to get that the restriction
$$
f:f^{-1}\big((a,b)\big)\longrightarrow (a,b)
$$
fibers $f^{-1}\big((a,b)\big)$ over $(a,b)$ topologically locally trivially in the stratified sense, so we have a trivialization homeomorphism
\be\label{prod}
\varphi:f^{-1}\left(c\right)\times (a,b)\longrightarrow f^{-1}\big((a,b)\big),
\ee
for $c \in (a,b)$.

The fiber $f^{-1}(c)$ is a compact real analytic subvariety and a pseudomanifold with only even codimension strata, so if we fix a given PL-structure of it, its corresponding intersection chain complex will consist of finitely generated abelian groups. In particular
$$
IH_i\left(f^{-1}(c)\right)
$$
are finitely generated abelian groups for all $i$. Since $\varphi$ is a homeomorphism, it follows from the topological invariance of intersection homology groups together with the K\"unneth formula that
$$
IH_i\left(f^{-1}\big((a,b)\right)\big)\cong IH_i\left(f^{-1}(c)\right)
$$
for all $i$, thus finishing our proof.
\end{proof}

We can now prove (\ref{e3}) in the abelian context.
\begin{theorem}\label{thmmain}
Let $X$ be a pure-dimensional subvariety of complex dimension $n$ of an abelian variety $G$. Then
$$
(-1)^n \cdot I\chi(X)\geq 0
$$
\end{theorem}
\begin{proof} Since $I\chi$ is additive under disjoint union and connected components of projective varieties are again projective, we may assume without any loss of generality that $X$ is connected.
We also assume $n > 0$, because the zero-dimensional case is trivial.

Let $f:X\longrightarrow S^1$ be a circle valued Morse function as in \cref{thmexistence}, and $T$ an associated open regular neighborhood of $X$ in $G$ for which $f$ extends to a map $\widetilde{f}: T \rightarrow S^1$.
Recall that $X \hookrightarrow T$ is a homotopy equivalence, and $f$ induces a surjection on first homology, hence also on fundamental groups.

Let $\widetilde{X}$ and $\widetilde{T}$ be the respective infinite cyclic covers of $X$ and $T$ associated to $\ker(f_*)$ and $\ker({\widetilde f}_*)$. We can naturally regard $\widetilde X$ as a Whitney stratified complex analytic subvariety of the smooth complex analytic manifold $\widetilde T$. Let $\R\longrightarrow S^1$ be the universal cover. We have an induced map $f':\widetilde X \longrightarrow \R$ that factors through $\widetilde T$ by a smooth map $\widetilde f':\widetilde T\longrightarrow \R$. While $\widetilde X$ is not compact, the map $f'$ is a proper map. Moreover, $f':\widetilde X \longrightarrow \R$ is a Morse function such that the stratified Morse index of any of its critical points is $n$.
By \cref{corfg}, in order to prove the theorem in this case, it suffices to show that $IH_i(\widetilde X)$ is a finitely generated abelian group for all $i\neq n$.

Assume first that $f$ has at least one critical point. Let $\{c_k\}\subset\R$ be the critical values of $f'$, which are the lifts of the critical values of $f$ to the universal cover $\R$, ordered in an increasing order, for $k\in\Z$. Let $\delta>0$ be small enough so that $c_{k+1}-c_k>\delta$ for all $k$. Let $b_k=c_k-\delta$, for all $k$.

We let $\widetilde X_{< a}:=(f')^{-1}((-\infty,a))$, and  $\widetilde X_{(a,b)}:=(f')^{-1}((a,b))$ for all $a<b$. Since $b_0$ is not a critical value of $f'$, there exists $\frac{1}{2}>\vepsi>0$ small enough such that there are no critical values of $f'$ in $[b_0-\vepsi,b_0+\vepsi]$ and
$$\widetilde X_{(b_0-\vepsi,b_0+\vepsi)}\cong f^{-1}\left((b_0-\vepsi,b_0+\vepsi)\right),$$
where $(b_0-\vepsi,b_0+\vepsi)$ is naturally seen here inside of $S^1$. It then follows by  \cref{lemfiniteIH} that
$$
IH_i(\widetilde X_{(b_0-\vepsi,b_0+\vepsi)})
$$
are finitely generated abelian groups for all $i$.

Let $Z:=(f')^{-1}\left((-\infty,b_0+\vepsi)\right)$, with its induced Whitney stratification inside  the manifold $M:=(\widetilde f')^{-1}\left((-\infty,b_0+\vepsi)\right)$. Consider the function
$$
\widetilde g:M\longrightarrow \R
$$
given by $\widetilde g=-\widetilde f'$, and let $g:Z\longrightarrow \R$ be the restriction of $\widetilde g$ to $Z$.

Let $a_0:=-b_0+\vepsi$ and $a_k=-b_{-k}$ for all $k>0$. We have the following long exact sequence of a pair in intersection homology, for all $k\geq 0$:
$$
\begin{array}{c}
\ldots\lra IH_{i+1}(g^{-1}\left((-\infty,a_{k+1})\right),g^{-1}\left((-\infty,a_{k})\right))\lra IH_i(g^{-1}\left((-\infty,a_{k})\right))\lra\\ \lra IH_i(g^{-1}\left((-\infty,a_{k+1})\right))
\lra IH_i(g^{-1}\left((-\infty,a_{k+1})\right),g^{-1}\left((-\infty,a_{k})\right))\lra\ldots
\end{array}
$$
Using \cref{thmvariation}, \cref{relinthom}, and the fact that the stratified Morse index of every critical point of $f'$ is $n$, we get that
$$
\left\{\begin{array}{lr}
IH_i(g^{-1}\left((-\infty,a_{k})\right))\cong IH_i(g^{-1}\left((-\infty,a_{k+1})\right)) &\text{ for }i\neq n-1, n\\
IH_i(g^{-1}\left((-\infty,a_{k})\right))\twoheadrightarrow IH_i(g^{-1}\left((-\infty,a_{k+1})\right)) &\text{ for }i=n-1\\
IH_i(g^{-1}\left((-\infty,a_{k})\right))\hookrightarrow IH_i(g^{-1}\left((-\infty,a_{k+1})\right)) &\text{ for }i=n\\
\end{array}\right.
$$
for every $k\geq 0$, or equivalently
\begin{equation}
\left\{\begin{array}{lr}
IH_i(\widetilde{X}_{(-a_k,b_0+\vepsi)})\cong IH_i(\widetilde{X}_{(-a_{k+1},b_0+\vepsi)}) &\text{ for }i\neq n-1, n\\
IH_i(\widetilde{X}_{(-a_k,b_0+\vepsi)})\twoheadrightarrow IH_i(\widetilde{X}_{(-a_{k+1},b_0+\vepsi)}) &\text{ for }i=n-1\\
IH_i(\widetilde{X}_{(-a_k,b_0+\vepsi)})\hookrightarrow IH_i(\widetilde{X}_{(-a_{k+1},b_0+\vepsi)}) &\text{ for }i=n.\\
\end{array}\right.
\label{eqnarrows}
\end{equation}
Since $\bigcup_{k=0}^\infty \widetilde{X}_{(-a_k,b_0+\vepsi)}=\widetilde{X}_{<b_0+\vepsi}$, and every compact set of $\widetilde{X}_{<b_0+\vepsi}$ is contained in $\widetilde{X}_{(-a_k,b_0+\vepsi)}$ for some $k$, we have that
$$
\varinjlim_k IH_i(\widetilde{X}_{(-a_k,b_0+\vepsi)})=IH_i(\widetilde{X}_{<b_0+\vepsi}).
$$
Recall now that $IH_i(\widetilde X_{(b_0-\vepsi,b_0+\vepsi)})$ is a finitely generated abelian group for all $i$. On the other hand, the first two lines of (\ref{eqnarrows}) yield that
$$
\left\{\begin{array}{lr}
IH_i(\widetilde{X}_{<b_0+\vepsi})\cong IH_i(\widetilde X_{(-a_{0},b_0+\vepsi)})=IH_i(\widetilde X_{(b_0-\vepsi,b_0+\vepsi)}) &\text{ for }i\neq n-1, n\\
\dim_\C IH_i(\widetilde{X}_{<b_0+\vepsi})\leq \dim_\C IH_i(\widetilde X_{(b_0-\vepsi,b_0+\vepsi)}) &\text{ for }i=n-1.\\
\end{array}\right.
$$
Therefore, 
$$
IH_i(\widetilde{X}_{<b_0+\vepsi})
$$
is a finitely generated abelian group for all $i\neq n$.

Now, let $d_0:= b_0+\vepsi$ and $d_k=b_k$ for all $k>0$. We can induct on $k\geq 0$ as before, by using the following long exact sequence of pairs
$$
\ldots\rightarrow IH_{i+1}(\widetilde{X}_{<d_{k+1}},\widetilde{X}_{<d_{k}})\lra IH_i(\widetilde{X}_{<d_{k}})\lra IH_i(\widetilde{X}_{<d_{k+1}})\lra IH_i(\widetilde{X}_{<d_{k+1}},\widetilde{X}_{<d_{k}})\rightarrow\ldots ,
$$
and apply \cref{thmvariation} (for $f'$), \cref{relinthom}, and the fact that the stratified Morse index of every critical point of $f'$ is $n$ as before, to get that
$$
IH_i(\widetilde{X})\text{ is a finitely generated abelian group for all }i\neq n\text{.} 
$$
The result now follows from \cref{corfg}.

If $f$ does not have any critical points, we argue as follows. Let $a\in\R$, and let $\frac{1}{2}>\vepsi>0$. Let $Z=(f')^{-1}\big((-\infty,a+\vepsi)\big)$, with its induced Whitney stratification. We first apply Theorem \ref{ta} to the function $g=-f'_{|Z}$, and use the topological invariance of intersection homology to get that
$$
IH_i(\widetilde{X}_{<a+\vepsi})\cong IH_i(\widetilde X_{(a-\vepsi,a+\vepsi)})
$$
for all $i$.  We the apply Theorem \ref{ta} to the function $f'$ to conclude (by using again  the topological invariance of intersection homology) that 
$$
IH_i(\widetilde{X})\cong IH_i(\widetilde{X}_{<a+\vepsi})
$$
for all $i$. The result follows now, as before, from \cref{lemfiniteIH} and \cref{corfg}.

\end{proof}

The identity (\ref{e2}) for the Euler-Poincar\'e characteristic of local complete intersections is proved in the following result:

\begin{theorem}
\label{thmlci}
Let $X$ be a pure dimensional subvariety of complex dimension $n$ of an abelian variety $G$. If $X$ is a local complete intersection, then
$$
(-1)^n \cdot \chi(X)\geq 0.
$$
\end{theorem}

\begin{proof}
Again we may assume connectivity and nonzero dimension of $X$.
Let $f:X\longrightarrow S^1$ be a circle-valued Morse function as in \cref{thmexistence}. The proof follows the steps of the proof of \cref{thmmain}. 
We just need to observe that the complex links around critical points of the real-valued Morse function induced on the infinite cyclic cover $\widetilde X$ are homotopy equivalent to the corresponding complex links of strata in $X$ via the local homeomorphism given by the covering $\widetilde T\longrightarrow T$. Hence, since $X$ is a local complete intersection, by (\ref{lci}) in Remark \ref{clv}, we get that
$$
H_i(\widetilde X_{\leq b}, \widetilde X_{\leq a})\cong 
\left\{\begin{array}{cc}
\Z^{\oplus r} & \text{if }i=n\\
0 & \text{if }i\neq n
\end{array}\right.
$$
for some integer $r$, 
if $a<b$ are such that there is only one critical point $x$ of $f'$ in $(f')^{-1}([a,b])$, and $f'(x)\in(a,b)$. The rest of the proof follows as in the proof of \cref{thmmain}, using also \cref{remusual}.
\end{proof}


\providecommand{\bysame}{\leavevmode\hbox
to3em{\hrulefill}\thinspace}

\end{document}